\documentclass[12pt,reqno]{article}
\usepackage{amsthm}
\usepackage{amsfonts}
\usepackage{amsmath}
\usepackage{mathrsfs}
\usepackage[usenames]{color}

\usepackage[colorlinks=true,
linkcolor=webgreen,
filecolor=webbrown,
citecolor=webgreen]{hyperref}

\definecolor{webgreen}{rgb}{0,.5,0}
\definecolor{webbrown}{rgb}{.6,0,0}

\usepackage{color}

\def\N{\mathbb{N}}
\def\M{\mathscr{M}}
\def\mod{\mbox{mod}}
\theoremstyle{definition}
\newtheorem{defn}{Definition}
\newtheorem{theorem}{Theorem}
\newtheorem{lemma}[theorem]{Lemma}
\newtheorem{example}[theorem]{Example}
\newtheorem{corollary}[theorem]{Corollary}
\newtheorem{proposition}[theorem]{Proposition}
\newtheorem{remark}{Remark}
\begin{document}

\title{On guessing whether a sequence has a certain property}

\author{Samuel Alexander\\ The Ohio State University\\ Department of Mathematics\\ 231 West 18th Avenue,\\ Columbus OH 43210\\ USA\\
\href{mailto:alexander@math.ohio-state.edu}{\tt alexander@math.ohio-state.edu}}
\date{}

\maketitle

\begin{abstract}
A concept of ``guessability'' is defined for sets of sequences of naturals.
Eventually, these sets are thoroughly characterized.
To do this, a nonstandard logic is developed, a logic containing symbols for the ellipsis as well
as for functions without fixed arity.
\end{abstract}

\section{Motivation}

Suppose Alice and Bob are playing a game.  Alice is reading a fixed sequence, one entry at a time.
Bob is trying to guess whether 0 is in the sequence.  He can revise his guess with each new revealed
entry, and he wins if his guesses converge to the correct answer.  He has an obvious strategy: always
guess no, until 0 appears (if ever), then guess yes forever.  The set of
sequences containing 0 is guessable.

Suppose, instead, Bob is trying to guess whether Alice's sequence contains \emph{infinitely many} zeroes.
We will see there is no strategy, not even if Bob has unlimited computation power.  The set of sequences
with infinitely many zeroes is unguessable.

A sequence $f:\N\to\N$ is \emph{onto} if $\forall m\,\exists n\,f(n)=m$.  This definition uses
nested quantifiers: quantifiers appear in the scope of other quantifiers.  Is it possible to give an alternate 
definition without nested quantifiers?  The answer is ``no'', but how to prove it?  We will give a proof of a very
strong negative answer, strong in the sense that nested quantifiers cannot be eliminated even in an
extremely powerful language.  Of course, the technique generalizes to a wide class of sets of sequences, not just the onto 
sequences.

\section{Basics}

Let $\N^{\N}$ be the set of sequences $f:\N\to\N$, and let $\N^{<\N}$ be the set of finite sequences.

\begin{defn}
A function $G:\N^{<\N}\to\{0,1\}$ \emph{guesses} (and is a \emph{guesser} for)
a set $S\subseteq\N^{\N}$ if for every $f:\N\to\N$, there exists some $m>0$ such that for all $n>m$,
\[
G(f(0),\ldots,f(n))=\begin{cases} 1, & \text{if $f\in S$;}\\ 0, & \text{if $f\not\in 
S.$}\end{cases}
\]
A set $S\subseteq\N^{\N}$ is \emph{guessable} if it has a guesser.
\end{defn}

The next test is very useful for showing nonguessability, though its converse is not true.

\begin{theorem}
\label{2.2}
Let $S\subseteq\N^{\N}$.  Suppose that, for every finite sequence $g\in\N^{<\N}$,
there are sequences $g_1,g_2\in \N^{\N}$ extending $g$ with $g_1\in S$ and $g_2\not\in S$.  
Then $S$ is nonguessable.
\end{theorem}

\begin{proof}
Suppose $S$ has a guesser $G$.  I will define a sequence $f:\N\to\N$ such that
$G(f(0),\ldots,f(n))$ fails to converge, which violates the definition of guesser.

Clearly $S\not=\emptyset$, so let $s_1:\N\to\N$ be some sequence in $S$.  By definition of guesser, we can find
some $x_1$ such that $G(s_1(0),\ldots,s_1(x_1))=1$.  Let $f(0)=s_1(0),\ldots,f(x_1)=s_1(x_1)$.

Inductively, suppose $x_1<\cdots<x_k$ and $f(0),\ldots,f(x_k)$ are defined such that 
$G(f(0),\ldots,f(x_i))\equiv 
i\,(\mod\,2)$
for $i=1,\ldots,k$.
By the theorem's hypothesis, we can find some $s_{k+1}:\N\to\N$, extending the finite sequence 
$(f(0),\ldots,f(x_k))$,
such that $s_{k+1}$ is in $S$ iff $k+1\equiv 1\,(\mod\,2)$.
By definition of guesser, find $x_{k+1}>x_k$ such that $G(s_{k+1}(0),\ldots,s_{k+1}(x_{k+1}))\equiv 
k+1\,(\mod\,2)$.
Let $f(x_k+1)=s_{k+1}(x_k+1),\ldots,f(x_{k+1})=s_{k+1}(x_{k+1})$.

This defines sequences $f:\N\to\N$ and $x_1<x_2<\cdots$ with the property that $G(f(0),\ldots,f(x_i))\equiv 
i\,(\mod\,2)$ for every $i>0$.
This contradicts that $G(f(0),\ldots,f(n))$ is supposed to converge.
\end{proof}

Using the above test, we can immediately confirm, for example, the set of sequences containing infinitely many zeroes is nonguessable, as is
the set of onto sequences.

\begin{remark}
Theorem~\ref{2.2} is constructive up to certain choices.  Starting with a set $S$ satisfying the hypotheses of 
Theorem~\ref{2.2} and naively 
trying to guess it, and being systematic in the choices from the proof,
can lead to the creation of a concrete sequence which thwarts the naive guessing attempt.
In an informal sense, it should be especially difficult for someone not in the know
to guess whether the resulting sequence lies in $S$.  And the more sophisticated the futile guessing attempt,
the more difficult the resulting sequence becomes.  For some explicit examples, see sequences A082691, 
A182659, and A182660
in Sloane's OEIS \cite{OEIS}.\end{remark}

Tsaban and Zdomskyy also briefly mention a somewhat similar notion of guessable sets in their paper 
\cite{tsaban}.

\section{A Logic for Ellipses}

Because guessers are functions which do not have ``arity'' in the usual sense, instead being defined on the whole space $\N^{<\N}$ of finite sequences,
and since we care so much about expressions like $G(f(0),\ldots,f(n))$, we will extend logic to mesh better with 
these sorts of expressions.  I assume familiarity with basic first-order logic, which 
Enderton \cite{enderton} has 
written about extensively, as has Bilaniuk \cite{bilaniuk}.

\begin{defn}
A \emph{language with ellipses} is a standard language of first-order logic,
with a constant symbol $\mathbf{0}$,
together with a set of \emph{function symbols of arity $\N^{<\N}$} and a special logical symbol $\cdots_x$ for every variable $x$.
\end{defn}

To avoid confusion, we will use $\cdots_x$ for the syntactical symbol and $\ldots$ for meta-ellipses.  For example, 
$G(s(\mathbf{0}),\ldots,s(\mathbf{2}))$
is a meta-abbreviation for \[G(s(\mathbf{0}),s(\mathbf{1}),s(\mathbf{2})),\] different than $G(s(\mathbf{0}),\cdots_x,s(\mathbf{2}))$ which has 
no counterpart in classical logic.

\begin{defn}If $\mathscr{L}$ is a language with ellipses, then the \emph{terms} of $\mathscr{L}$ 
(and their free variables) are defined
inductively:
\begin{enumerate}
\item For any variable $x$, $x$ is a term and $FV(x)=\{x\}$.
\item For any constant symbol $c$, $c$ is a term and $FV(c)=\emptyset$.
\item If $f$ is a function symbol of arity $n$ or arity $\N^{<\N}$, and $t_1,\ldots,t_n$ are terms, then 
$f(t_1,\ldots,t_n)$ is a term with free vars $FV(t_1)\cup\cdots\cup FV(t_n)$.
\item If $G$ is an $\N^{<\N}$-ary function symbol, and $u,v$ are terms, and $x$ is a variable, then 
$G(u(\mathbf{0}),\cdots_x,u(v))$ is a term with
free variables 
\[(FV(u)\backslash \{x\})\cup FV(v).\]
\end{enumerate}
The \emph{well-formed formulas} of $\mathscr{L}$ are defined as usual from these terms.
Term substitution is defined by the usual induction with two new cases:
\begin{itemize}
\item If $y\not=x$ then \[G(u(\mathbf{0}),\cdots_x,u(v))(y|t)=G(u(y|t)(\mathbf{0}),\cdots_x,u(y|t)(v(y|t))).\]
\item $G(u(\mathbf{0}),\cdots_x,u(v))(x|t)=G(u(\mathbf{0}),\cdots_x,u(v(x|t)))$.
\end{itemize}

A \emph{model} for a language with ellipses $\mathscr{L}$ is a model $\M$ for the classical part of $\mathscr{L}$, together with
a function $G^{\M}:\M^{<\N}\to\M$ for each $\N^{<\N}$-ary function symbol $G$ in $\mathscr{L}$.
However, defining how an arbitrary model evaluates terms is difficult.  We will only be interested in
one very specific family of models, where there is no trouble evaluating terms.
\end{defn}

\begin{defn}
The following models lie at the heart of all later results.

\begin{itemize}
\item $\mathscr{L}_{\mbox{max}}$ is the language with ellipses which contains a constant symbol $\mathbf{n}$ for
every $n\in\N$,
an $n$-ary function symbol $\tilde{w}$ for every function $w:\N^n\to\N$ ($n>0$),
an $n$-ary predicate symbol $\tilde{p}$ for every subset $p\subseteq \N^n$ ($n>0$),
an $\N^{<\N}$-ary function symbol $\tilde{G}$ for every function $G:\N^{<\N}\to\N$,
and one additional unary function symbol $\mathbf{f}$.
\item For every function $f:\N\to\N$, $\M_f$ is the model for the language $\mathscr{L}_{\mbox{max}}$ with
universe $\N$, which interprets $\mathbf{n}$ as $n$ for every $n$, $\tilde{w}$ as $w$
for every $w:\N^n\to\N$, $\tilde{p}$ as $p$ for every $p\subseteq\N^n$,
and $\tilde{G}$ as $G$ for every $G:\N^{<\N}\to\N$, and which interprets $\mathbf{f}$
as $f$.
\end{itemize}
\end{defn}

If $n\in\N$ then $\bar{n}$ denotes the numeral $\mathbf{n}$ of $n$.

\begin{defn}
Let $f:\N\to\N$.  The semantics of $\M_f$ are defined as follows.  Let $s$ be any assignment from the 
variables
to $\N$.
\begin{itemize}
\item $(\M_f,s)$ interprets terms $t$ into naturals $t^{\M_f,s}$, or $t^s$ if there is no ambiguity,
according to the usual inductive definition,
with one new case:
\begin{itemize}
\item If $u,v$ are terms and $x$ is a variable and $G$ is an $\N^{<\N}$-ary function symbol,
then
\[
G(u(\mathbf{0}),\cdots_x,u(v))^{s}=G^{\M_f}\left(u(x|\mathbf{0})^{s},\ldots, 
u\left(x\left|\overline{v^{s}}\right.\right)^{s}\right).\]
\end{itemize}

\item
For example, the interpretation of $\tilde{G}(\mathbf{f}(x)(\mathbf{0}),\cdots_x,\mathbf{f}(x)(\mathbf{99}))$ is
\[G(f(0),\ldots,f(99)),\]
while the interpretation of $\tilde{G}(\mathbf{f}(x)(\mathbf{0}),\cdots_x,\mathbf{f}(x)(\mathbf{f}(y)))$
is \[G(f(0),\ldots,f(f(y^{s}))).\]

\item From here, the remaining semantics of $\M_f$ are defined as usual.
\end{itemize}
\end{defn}

In classical logic, every term with no free variables has the property that
its interpretation in any model depends only on finitely many values of
the interpretations of the function symbols in that model.
For example, the interpretation of $5+(2\cdot 3)$ depends only on one value of $\cdot$ and
one value of $+$.
Similar properties are true of our $\M_f$ models.

\begin{lemma}
\label{3.5}
Suppose $u$ is a term with no free variables, and $c$ is a constant symbol.
For any $f:\N\to\N$, $\M_f\models u=c$ iff there is some $k$ such that
whenever $g:\N\to\N$ extends $(f(0),\ldots,f(k))$, $\M_g\models u=c$ and to check
whether $\M_g\models u=c$ using the inductive definition of semantics for $\M_g$,
it is not necessary to query $g(i)$ for any $i>k$.
\end{lemma}

\begin{proof}
($\Rightarrow$)  Induction on complexity of $u$.

\begin{itemize}
\item Since $u$ has no free variables, $u$ cannot be a variable.  If $u$ is a constant symbol, the lemma is trivial.

\item Suppose that $u$ is $h(u_1,\ldots,u_n)$ for some $n$-ary (or $\N^{<\N}$-ary) function symbol $h$ other than 
$\mathbf{f}$, and some
terms $u_1,\ldots,u_n$
with no free variables.
If $\M_f\models h(u_1,\ldots,u_n)=c$, then there are $a_1,\ldots,a_n\in \N$ such that 
$h^{\M_f}(a_1,\ldots,a_n)=c^{\M_f}$
and $\M_f\models u_i=\bar{a_i}$ for $i=1,\ldots,n$.
Since $\bar{a_i}$ is a constant symbol, by induction find $k_1,\ldots,k_n$ such that for any $i=1,\ldots,n$ and
any $g:\N\to\N$ with $g(0)=f(0),\ldots,g(k_i)=f(k_i)$, $\M_g\models u_i=\bar{a_i}$, and checking this by 
definition of
semantics does not require querying $g(j)$ for any $j>a_i$.
Then $k=\max\{k_1,\ldots,k_n\}$ works (using the fact $h^{\M_g}$ does not depend on $g$ since $h$ is not 
$\mathbf{f}$).

\item Next, suppose $u$ is $\mathbf{f}(v)$ where $v$ is a term with no free variables.
If $\M_f\models \mathbf{f}(v)=c$ then there iss $a\in\N$ such that $f(a)=c^{\M_f}$ and $\M_f\models v=\bar{a}$.
Since $\bar{a}$ is a constant symbol, by induction find $k_0$ such that whenever 
$g(0)=f(0),\ldots,g(k_0)=f(k_0)$,
then $\M_g\models v=\bar{a}$, and checking $\M_g\models v=\bar{a}$ does not require querying $g(i)$ for any 
$i>k_0$.
Let $k=\max\{k_0,a\}$.
Suppose $g(0)=f(0),\ldots,g(k)=f(k)$.  Then 
$\mathbf{f}(v)^{\M_g}=\mathbf{f}^{\M_g}(v^{\M_g})=g(a)=f(a)=c^{\M_f}$.
So $\M_g\models \mathbf{f}(v)=c$, and to check so, we only had to query $g(a)$ in addition to any queries we
had to make to check $\M_g\models v=\bar{a}$, so we did not have to query $g(i)$ for any $i>k$.

\item Finally, suppose $u$ is $G(v(\mathbf{0}),\cdots_x,v(w))$ where $v,w$ are terms, $x$ is a variable, $FV(w)=\emptyset$, 
$FV(v)\subseteq \{x\}$, and $G$ is an $\N^{<\N}$-ary function symbol.
If $\M_f\models G(v(\mathbf{0}),\cdots_x,v(w))=c$ then
\[G^{\M_f}\left(v(x|\mathbf{0})^{\M_f},\ldots,v\left(x\left|\overline{w^{\M_f}}\right.\right)^{\M_f}\right)=c^{\M_f}.\]
Since $\M_f\models w=\overline{w^{\M_f}}$,
find some number $k_{-1}$ such that whenever $g$ extends $(f(0),...,f(k_{-1}))$, $\M_g\models w=\overline{w^{\M_f}}$
and checking so does not require queries beyond $g(k_{-1})$.
Since $\M_f\models v(x|\mathbf{i})=\overline{v(x|\mathbf{i})}$ for $i=0,\ldots,w^{\M_f}$,
find $k_0,\ldots,k_{w^{\M_f}}$ such that for each $i=0,\ldots,w^{\M_f}$, if $g(0)=f(0),\ldots,g(k_i)=f(k_i)$ then
$\M_g\models v(x|\overline{i})=\overline{v(x|\overline{i})}$ can be confirmed without querying $g$ beyond $g(k_i)$.

Let $k=\max\{k_{-1},k_0,\ldots,k_{w^{\M_f}}\}$.  Suppose $g(0)=f(0),\ldots,g(k)=f(k)$.
Then $\M_g\models w=\overline{w^{\M_f}}$, so $w^{\M_g}=w^{\M_f}$.
Similarly $v(x|\mathbf{i})^{\M_g}=v(x|\mathbf{i})^{\M_f}$ for $i=0,\ldots,w^{\M_f}$.
And $G^{\M_g}=G^{\M_f}$.  It follows that \[\M_g\models G(v(\mathbf{0}),\cdots_x,v(w))=c,\]
and checking so does not require any queries to $g(j)$ for any $j>k$.
\end{itemize}

($\Leftarrow$)  Suppose there is some $k$ so that whenever $g$ extends $(f(0),\ldots,f(k))$ then $\M_g\models 
u=c$.
In particular, $f$ itself extends $(f(0),\ldots,f(k))$, so $\M_f\models u=c$.
\end{proof}

\begin{corollary}
\label{3.6}
Let $\phi$ be a quantifier-free sentence.
For any $f:\N\to\N$, $\M_f\models \phi$ iff there is some $k$
such that for every $g:\N\to\N$ extending $(f(0),\ldots,f(k))$, $\M_g\models \phi$,
and in checking $\M_g\models\phi$ by the inductive definition of semantics,
we never need to query $g(i)$ for any $i>k$.
\end{corollary}

\begin{proof}
By induction on the complexity of $\phi$.
\begin{itemize}
\item Suppose $\phi$ is $u=v$ for terms $u,v$ with no free variables.
Assume $\M_f\models u=v$.  Then $\M_f\models u=\overline{u^{\M_f}}$ and $\M_f\models v=\overline{u^{\M_f}}$.
By Lemma~\ref{3.5}, find $k$ big enough that whenever $g(0)=f(0),\ldots,g(k)=f(k)$, then $\M_g\models 
u=\overline{u^{\M_f}}$
and $\M_g\models v=\overline{u^{\M_f}}$ and both facts can be confirmed without
querying $g$ beyond $g(k)$.  For any such $g$, $\M_g\models u=v$, verifiable with no additional $g$-queries.  The converse is trivial.

\item Next, suppose $\phi$ is $\tilde{p}(u_1,\ldots,u_n)$ for an $n$-ary predicate symbol $\tilde{p}$ and terms 
$u_1,\ldots,u_n$ with no free variables.
Then $\phi$ is equivalent (in every $\M_{\cdot}$) to $\tilde{g}(u_1,\ldots,u_n)=\mathbf{1}$ where $g$ is the 
characteristic function
of $p$, so we are done by the previous case.

\item Suppose $\phi$ is $\phi_1\wedge\phi_2$.  Assume $\M_f\models \phi$.  Inductively, find $k_1$ and $k_2$ such 
that if $g$ extends $(f(0),\ldots,f(k_i))$
then $\M_g\models \phi_i$ is verifiable
with no $g$-queries beyond $g(k_i)$.  Then any $g$ extending $(f(0),\ldots,f(\max\{k_1,k_2\}))$ has $\M_g\models 
\phi$,
verifiable without querying beyond $g(\max\{k_1,k_2\})$.  The converse is trivial.

\item The cases of other propositional connectives are similar.
\end{itemize}
\end{proof}

If $s$ is an assignment from the variables of a language onto the universe of the language, and if $x$ is a variable,
and $n$ is a number, then $s(x|n)$ denotes the assignment which is identical to $s$ except that it maps $x$ to $n$.
Similarly if a model is understood by context and $c$ is a constant symbol then $s(x|c)$ denotes the assignment identical to $s$ except that it maps 
$x$
to the interpretation of $c$ in the model.

\begin{lemma}
\label{3.7}
\emph{(The Weak Substitution Lemma)}
For a formula $\phi$, an assignment $s$, and a constant symbol $c$, and 
for any $f:\N\to\N$, $\M_f\models \phi[s(x|\overline{c^{s}})]$ iff $\M_f\models \phi(x|c)[s]$.
\end{lemma}

\begin{proof}
By the inductive argument used to prove the full Substitution Lemma in classical logic, most of which we omit.
But there are tricky new cases for our new
terms.

\textbf{Claim:} For any terms $u,v$, constant symbol $c$, variables $x\not=y$, and assignment $s$,
\[
G(u(\mathbf{0}),\cdots_x,u(v))(y|c)^s = G(u(\mathbf{0}),\cdots_x,u(v))^{s(y|c)}.\]
The details are (using the induction hypothesis repeatedly) as follows:
\begin{align*}
G(u(\mathbf{0}),\cdots_x,u(v))(y|c)^s
&=
G\left(u(y|c)(\mathbf{0}),\cdots_x,u(y|c)(v(y|c))\right)^s\\
&=
G^{\M_f}\left(u(y|c)(x|\mathbf{0})^s,\ldots,u(y|c)\left(x\left|\overline{v(y|c)^s}\right.\right)^s\right)\\
&=
G^{\M_f}\left(u(x|\mathbf{0})^{s(y|c)},\ldots,u\left(x\left|\overline{v^{s(y|c)}}\right.\right)^{s(y|c)}\right)\\
&=
G(u(\mathbf{0}),\cdots_x,u(v))^{s(y|c)}.\end{align*}

\textbf{Claim:} For any terms $u,v$, constant symbol $c$, and variable $x$ and assignment $s$,
\[
G(u(\mathbf{0}),\cdots_x,u(v))(x|c)^s = G(u(\mathbf{0}),\cdots_x,u(v))^{s(x|c)}.\]
Using the induction hypothesis repeatedly:
\begin{align*}
G(u(\mathbf{0}),\cdots_x,u(v))(x|c)^s
&=
G(u(\mathbf{0}),\cdots_x,u(v(x|c)))^s\\
&=
G^{\M_f}\left( u(x|\mathbf{0})^s,\ldots, u\left(x\left|\overline{v(x|c)^s}\right.\right)^s\right)\\
&=
G^{\M_f}\left( u(x|\mathbf{0})^s,\ldots,
u\left(x\left|\overline{v^{s(x|c)}}\right.\right)^s\right)\\
&=
G^{\M_f}\left( u(x|\mathbf{0})^{s(x|c)},\ldots,
u\left(x\left|\overline{v^{s(x|c)}}\right.\right)^{s(x|c)}\right)\\
&=
G(u(\mathbf{0}),\cdots_x,u(v))^{s(x|c)}.\end{align*}
The next to last equation is justified because the terms whose ``exponents'' are changed do not depend on $x$.
\end{proof}

A full Substitution Lemma is also true, but it requires a nonclassical definition of \emph{substitutable}, which 
would take us too far afield.

\section{Guessability and Quantifiers}

\begin{defn}
Let $S\subseteq \N^{\N}$ be a set of sequences.  Let $\phi$ be a sentence in $\mathscr{L}_{\mbox{max}}$.
We say that $\phi$ \emph{defines} $S$ if, for every $f:\N\to\N$, $\M_f\models \phi$ iff $f\in S$.
\end{defn}

\begin{theorem}
\label{4.7}
A set $S\subseteq\N^{\N}$ is guessable if and only if it is defined
by some sentence $\forall x\,\exists y\,\phi$ and also by some sentence $\exists x\,\forall y\,\psi$,
where $\phi$ and $\psi$ are quantifier-free.
\end{theorem}

We
divide the proof of the theorem above into a sequence of lemmata.

\begin{lemma}
\label{4.2}
Suppose $S\subseteq \N^{\N}$ is guessable.  Then $S$ is defined by 
some sentence $\exists x\,\forall y\,\phi$ and also
by some sentence $\forall x\,\exists y\,\psi$, where $\phi$ and $\psi$ are quantifier-free.
\end{lemma}

\begin{proof}
Let $G$ be a guesser for $S$.
For any $f:\N\to\N$, $f\in S$ if and only if 
$G(f(0),\ldots,f(n))=1$ for all $n$ sufficiently large.
Therefore $S$ is defined by
\[
\exists x\, \forall y\, ((y>x)\rightarrow \tilde{G}(\mathbf{f}(z)(\mathbf{0}),\cdots_z,\mathbf{f}(z)(y))=\mathbf{1}),\]
where ``$y>x$'' is shorthand for $\tilde{>}(y,x)$.
Similarly, $S$ is also defined by
\[
\forall x\, \exists y\, ((y>x)\wedge \tilde{G}(\mathbf{f}(z)(\mathbf{0}),\cdots_z,\mathbf{f}(z)(y))=\mathbf{1}).\]
\end{proof}

We will prove the converse of Lemma~\ref{4.2} shortly.  To that end, a piece of technical machinery is needed.

\begin{defn}
A set $S\subseteq \N^{\N}$ is \emph{overguessable} if there is a function
$\mu:\N^{<\N}\to\N\cup\{\infty\}$ such that:
\begin{enumerate}
\item For every $f\in S$, $\mu(f(0),\ldots,f(n))$ is eventually bounded by a finite number.
\item For every $f\not\in S$, $\mu(f(0),\ldots,f(n))\to\infty$ as $n\to\infty$.
\end{enumerate}
\end{defn}

\begin{lemma}
\label{4.4}
Suppose $S\subseteq \N^{\N}$ is defined by the sentence $\exists x\,\forall y\,\phi$ where $\phi$ is quantifier-free.
Then $S$ is overguessable.
\end{lemma}

\begin{proof}
Given a tuple $(n_0,\ldots,n_k)$, define $\mu(n_0,\ldots,n_k)$ as follows.
Let $h:\N\to\N$ be defined by $h(i)=n_i$ if $i\leq k$, $h(i)=0$ otherwise.
Given a pair $(a,b)\in\N^2$, consider the sentence $\phi(x,y|\bar{a},\bar{b})$.
Attempt to check whether $\M_h\models \phi(x,y|\bar{a},\bar{b})$, using the
inductive definition of the semantics of $\M_h$.  If, in so doing, you must
query $h(i)$ for some $i>k$, say that \emph{the attempt failed}.  Otherwise,
say \emph{the attempt succeeded}.  If the attempt failed, \emph{or} if
$\M_h\models \phi(x,y|\bar{a},\bar{b})$, then say that $(a,b)$ is \emph{nice}.

Call a number $a$ \emph{very nice} if $(a,b)$ is nice for every $b$.  If there is any very nice number,
then let $\mu(n_0,\ldots,n_k)$ be the smallest very nice number.  Otherwise let $\mu(n_0,\ldots,n_k)=\infty$.

I claim the above $\mu$ witnesses that $S$ is overguessable.

First, suppose $f\in S$.  Since $S$ is defined by $\exists x\,\forall y\,\phi$,
$\M_f\models \exists x\,\forall y\,\phi$.
By the Weak Substitution Lemma, for some $a$, $\M_f\models \phi(x,y|\bar{a},\bar{b})$ for every $b$.
When we attempt to check whether $\M_h\models \phi(x,y|\bar{a},\bar{b})$ in the definition of
$\mu(f(0),\ldots,f(k))$,
if the attempt succeeds, then $\M_h\models \phi(x,y|\bar{a},\bar{b})$ because $\M_f\models \phi(x,y|\bar{a},\bar{b})$
and we never had to look at the part of $h$ which disagrees with $f$.  So $(a,b)$ is nice for
every $b$, so $a$ is very nice, so $\mu(f(0),\ldots,f(k))$ is bounded by $a$.

Next, suppose $f\not\in S$.  Let $a\in\N$, I claim $\mu(f(0),...,f(n))\not=a$ for all $n$ sufficiently large. 
Since $f\not\in S$, $\M_f\not\models \exists x\,\forall y\,\phi$.
By the Weak Substitution Lemma, there is some $b$ such that $\M_f\not\models \phi(x,y|\bar{a},\bar{b})$.
Since $\phi$ is quantifier-free, we invoke Corollary~\ref{3.6} on $\neg\phi(x,y|\bar{a},\bar{b})$
and find $k$ 
such that
$\M_g\not\models \phi(x,y|\bar{a},\bar{b})$ whenever $g$ extends $(f(0),\ldots,f(k))$, and,
to check whether $\M_g\models \phi(x,y|\bar{a},\bar{b})$, we do not need to query $g(i)$
for $i>k$.  Then, in the definition of $\mu(f(0),\ldots,f(k))$, for pair $(a,b)$, the attempt succeeds
and $\M_h\not\models \phi(x,y|\bar{a},\bar{b})$, so $(a,b)$ is not nice, so $a$ is not very nice,
so $\mu(f(0),\ldots,f(k))\not=a$, in fact, $\mu(f(0),\ldots,f(j))\not=a$ for all $j\geq k$.
By arbitrariness of $a$, $\mu(f(0),...,f(n))\to\infty$.
\end{proof}

\begin{lemma}
\label{4.7oneway}
Suppose a set $S\subseteq\N^{\N}$ is
defined
by some sentence $\forall x\,\exists y\,\phi$ and also by some sentence $\exists x\,\forall y\,\psi$,
where $\phi$ and $\psi$ are quantifier-free.
Then $S$ is guessable.
\end{lemma}

\begin{proof}
By Lemma~\ref{4.4}, find $\mu:\N^{<\N}\to\N\cup\{\infty\}$ which overguesses $S$.  And since $S^c$ is defined by
$\exists x\,\forall y\,\neg\phi$, use Lemma~\ref{4.4} again to find $\nu:\N^{<\N}\to\N\cup\{\infty\}$ which
overguesses
$S^c$.

Define $G:\N^{<\N}\to\{0,1\}$ by saying $G(n_0,\ldots,n_k)=1$ if $\mu(n_0,\ldots,n_k)\leq \nu(n_0,\ldots,n_k)$
and
$0$
otherwise.
If $f\in S$ then $\mu(f(0),\ldots,f(k))$ is eventually bounded by a finite number and
$\nu(f(0),\ldots,f(k))\to\infty$,
so $G(f(0),\ldots,f(k))$ converges to $1$.  The other case is similar.
\end{proof}

Combining Lemmata~\ref{4.2} and \ref{4.7oneway} proves Theorem~\ref{4.7}.

\begin{proposition}
\label{4.5}
If $S\subset \N^{\N}$
is overguessable, then it is defined by some sentence $\exists x\forall y\phi$ with $\phi$ quantifier-free.
\end{proposition}

\begin{proof}
Suppose $S$ is overguessed by 
$\mu:\N^{<\N}\to\N\cup\{\infty\}$.
Define $\mu':\N^{\N}\to\N$ by saying $\mu'(n)=\mu(n)+1$ if $\mu(n)\not=\infty$, $\mu'(n)=0$ if $\mu(n)=\infty$.
If $f:\N\to\N$, then $f\in S$ if and only if the sequence
$\mu(f(0),\ldots,f(n))$ is eventually bounded by some finite number.
This is true if and only if $\mu'(f(0),\ldots,f(n))$ is eventually bounded by some finite number and eventually nonzero.
This latter equivalence can be expressed by
\[
\exists m_1\,\exists m_2\,\forall m_3\,((m_3>m_2)\rightarrow (0<\mu'(f(0),\ldots,f(m_3))<m_1)).
\]
Let $d:\N\to\N^2$ be any onto map from $\N$ to $\N^2$.
Write $d(n)=(d_1(n),d_2(n))$, thus defining two functions $d_1,d_2:\N\to\N$.
Then the above formula is equivalent to
\[
\exists m\,\forall m_3\,((m_3>d_2(m))\rightarrow (0<\mu'(f(0),\ldots,f(m_3))<d_1(m))).\]
This can be formalized in $\mathscr{L}_{\mbox{max}}$, providing a sentence $\exists x\,\forall y\,\phi$
which defines $S$, with $\phi$ quantifier-free.
\end{proof}

\begin{example}
\label{4.6}
Every countable subset of $\N^{\N}$ is overguessable.
\end{example}

\begin{proof}
Let $S\subseteq \N^{\N}$ be countable.
Define $g:\N^2\to\N$ by saying $g(m,n)=h_m(n)$ where $h_m$ is the $m$th element of $S$.
Then $S$ is defined by
\[
\exists x\,\forall y\,\tilde{g}(x,y)=\mathbf{f}(y).\]
By Lemma~\ref{4.4}, $S$ is overguessable.\end{proof}

\begin{remark}
\label{remark2}
Guessable and overguessable sets of sequences are analogous to computable and computably enumerable sets of naturals,
respectively.  One shows that $\Delta_1$ sets (in a much weaker logical setting than $\mathscr{L}_{\mbox{max}}$)
of naturals are computable by showing that they and their complements are c.e.~by using the characterization
of c.e.~sets as sets which are $\Sigma_1$-definable (in the weaker setting).  By comparison, I have shown that
$\Delta_2$ sets of sequences (in a very strong logical setting) are guessable by showing that they and their
complements are overguessable by using the characterization of overguessable sets as $\Sigma_2$-definable (in the
stronger setting).  These analogous phenomena in computability theory have been written about by Rogers 
\cite{rogers},
Enderton \cite{enderton}, Bilaniuk \cite{bilaniuk}, and many other authors.\end{remark}

We will elaborate more on Remark~\ref{remark2} in Section~\ref{dstsection}.

\begin{lemma}
\label{4.8}
Suppose $S\subseteq\N^{\N}$ is definable by a sentence $\forall x\,\exists y\,\phi$ where
$\phi$ is quantifier-free.  If $S$ is countable then $S$ is guessable.
\end{lemma}

\begin{proof}
Suppose $S$ is countable.  In the proof of Example~\ref{4.6}, we showed $S$ is 
definable by
a sentence $\exists x\,\forall y\,\psi$ where $\psi$ is quantifier-free.  By Theorem~\ref{4.7}, $S$ is 
guessable.\end{proof}

\begin{example}
\label{4.9}
There are uncountably many permutations of $\N$.
\end{example}

\begin{proof}
A function
$f:\N\to\N$ is a permutation iff 
\[
\forall m_1\,\forall m_2\,\exists n\,((f(m_1)=f(m_2)\rightarrow m_1=m_2)\wedge f(n)=m_2).\]
By appropriately coding $\langle m_1,m_2\rangle$, the set $S$ of permutations is
defined by a sentence $\forall x\,\exists y\,\phi$ where $\phi$ is quantifier-free.

Permutations are not guessable.  If $G$ were a permutation-guesser, it would diverge
on the following sequence.  Let $f(0)=0$, $f(1)=1$, and so on until $G(f(0),\ldots,f(k_1))=1$
(this must happen since $G$ would converge to $1$ if we kept going forever).
Then skip a number, $f(k_1+1)=k_1+2$, $f(k_1+2)=k_1+3$, and keep going until $G(f(0),\ldots,f(k_2))=0$.
Then fill in the gap, $f(k_2+1)=k_1+1$, and resume where we left off, $f(k_2+2)=k_2+2$, and so
on until $G(f(0),\ldots,f(k_3))=1$.  This process shows permutations are unguessable.

By Lemma~\ref{4.8}, $S$ is uncountable.
\end{proof}

\begin{example}
\emph{(Cantor)}
There are uncountably many real numbers.
\end{example}

\begin{proof}

Consider the set $A$ of numbers in the interval $(0,1)$ which have infinitely many $5$s in
their decimal expansions.  There is an obvious bijection between $A$ and the set $S$ of sequences
$f:\N\to\{0,9\}$ such that $f(n)=5$ infinitely often.
This set $S$ is defined by
\[\forall x\,\exists y\,((y>x)\wedge \mathbf{f}(y)=\mathbf{5}\wedge \mathbf{f}(x)\geq \mathbf{0} \wedge \mathbf{f}(x)\leq 
\mathbf{9}).\]
By 
Lemma~\ref{4.8}, if $S$ 
is countable then it is guessable.
But it is not:  if $G$ were a guesser for $S$, then we could define a sequence on whose initial segments $G$ does 
not converge.  Namely, let $f(0)=\ldots=f(x_k)=0$ where $x_k$ is big enough that $G(f(0),...,f(x_k))=0$, and then 
let
$f(x_k+1)=\cdots=f(x_{k+1})=5$, where $x_{k+1}>x_k$ is big enough that $G(f(0),...,f(x_{k+1}))=1$.  And so on, 
alternating, forever.  This shows $S$ is not guessable, so $S$ is not countable, so $A$ is uncountable, so 
$\mathbb{R}$ is uncountable.
\end{proof}

\begin{lemma}
\label{4.10}
If $S\subseteq \N^{\N}$ is definable by a sentence $\phi$
without nested quantifiers (that is, no quantifier appearing in the scope of another),
then $S$ is guessable.
\end{lemma}

\begin{proof}
If so, then $\phi$ is a propositional combination of quantifier-free
sentences and sentences of the form $\forall x\,\phi_0$ and $\exists x\,\phi_1$ where $\phi_0$, $\phi_1$
are quantifier-free.  The sets defined by these component sentences are guessable by Theorem~\ref{4.7}.
Clearly guessable sets are closed under union, intersection, and complement, so $S$ itself is
guessable.
\end{proof}

\begin{example}
\label{4.11}
The definition of onto functions cannot be simplified to get rid
of nested quantifiers, not even with the full power of $\mathscr{L}_{\mbox{max}}$.
\end{example}

\begin{proof}
By Lemma~\ref{4.10} and the fact the set of onto functions is not guessable, see 
Theorem~\ref{2.2}.
\end{proof}

\section{Descriptive Set Theory}
\label{dstsection}

In this section we will elaborate further on Remark~\ref{remark2}.
In descriptive set theory, $\N^{\N}$ is endowed with the topology whose basic open sets are those sets of the form
\[
\{f\in\N^{\N}\,:\,\mbox{$f$ extends $f_0$}\}\]
where $f_0\in\N^{<\N}$.  Since $\N^{<\N}$ is countable, $\N^{\N}$ is second countable in the sense of basic topology.
A set is called $G_{\delta}$ if it is a countable intersection of open sets, and $F_{\sigma}$ if it is a countable union
of closed sets.  A set is $\mathbf{\Delta}_2^0$ if it is both $G_{\delta}$ and $F_{\sigma}$ (equivalently, if it and its
complement are both $G_{\delta}$).  This $\mathbf{\Delta}_2^0$ is one of 
the levels of the \emph{Borel hierarchy} which many authors, including 
Moschovakis \cite{moschovakis}, have written about.

\begin{theorem}
Let $S\subseteq\N^{\N}$.  Then $S$ is guessable if and only if $S$ is $\mathbf{\Delta}_2^0$.
\end{theorem}

\begin{proof} ($\Rightarrow$)  Suppose $S$ is guessable.
By Lemma~\ref{4.2}, $S$ is defined by a sentence $\forall x\,\exists y\,\phi$ where $\phi$ is quantifier-free.
For every $i,j\in\N$, let $S_i\subseteq\N^{\N}$ be the set defined by $\exists y\,\phi(x|i)$ and let $T_{ij}\subseteq\N^{\N}$
be the set defined by $\phi(x,y|i,j)$.
It follows from Corollary~\ref{3.6} that each $T_{ij}$ is open.  Each $S_i=\cup_j T_{ij}$, so each $S_i$ is open.
Since $S=\cap_i S_i$, $S$ is $G_{\delta}$.  Since $S^c$ is also guessable, identical reasoning shows $S^c$ is $G_{\delta}$, so
$S$ is $\mathbf{\Delta}_2^0$.

($\Leftarrow$)
Suppose $S$ is $\mathbf{\Delta}_2^0$.  Write $S=\cap_{i\in\N}S_i$ where each $S_i$ is open.  By second countability,
write $S_i=\cup_{j\in\N}T_{ij}$ where each $T_{ij}$ is basic open.
By the nature of basic open sets of $\N^{\N}$, there are $T_{ij}^0\in\N^{<\N}$ such that each $T_{ij}$ is exactly the set
of infinite extensions of $T_{ij}^0$.
Let $\tau:\N^{<\N}\to\N$ be defined by saying $\tau(i,j,x_0,\ldots,x_k)=1$ if $(x_0,\ldots,x_k)=T_{ij}^0$, 
$\tau$ is $0$ everywhere else.
Define $\ell:\N^2\to\N$ by letting $\ell(i,j)$ be the length of $T_{ij}^0$.
Then for any $f:\N\to\N$, $f$ extends $T_{ij}^0$ if and only if $\tau(i,j,f(0),\ldots,f(\ell(i,j)))=1$.
Thus, $f\in S$ if and only if
\[\forall i\,\exists j\,\tau(i,j,f(0),\ldots,f(\ell(i,j)))=1.\]
This can be formalized in $\mathscr{L}_{\mbox{max}}$.
By dual reasoning applied to $S^c$, $S$ can also be defined by some $\exists i\,\forall j\,\phi$ where $\phi$ is quantifier-free.
By Theorem~\ref{4.7}, $S$ is guessable.
\end{proof}

\section{Addendum}

In January 2012, we learned that the notion of guessability
was introduced some time ago in the PhD dissertation of William
W.~Wadge \cite{wadge}.  Instead of considering guesser functions,
Wadge considered \emph{guesser sets}, calling a subset $S\subseteq
\N^{\N}$ guessable if there are disjoint sets $U,W\subseteq \N^{<\N}$
such that for every sequence $f$, $f\in S$ iff $f|k\in U$ for all but
finitely many $k$, and $f\not\in S$ iff $f|k\in W$ for all but finitely
many $k$.  This is clearly equivalent to our definition.
Wadge gave a game-theoretical proof that guessability is equivalent to
being $\mathbf{\Delta}_2^0$ (our Theorem 16) and then used this fact
to show a special case, for $\mathbf{\Delta}_2^0$ sets, of what is
now known as \emph{Wadge's lemma}, an important result about Wadge
degrees studied by descriptive set theorists.

%
%
%
%

\end{document}